\documentclass[11pt]{article}

\usepackage{epsfig,epsf,fancybox}
\usepackage{amsmath}
\usepackage{mathrsfs}
\usepackage{amssymb}
\usepackage{graphicx}
\usepackage{color}

\newtheorem{theorem}{Theorem}[section]

\newtheorem{lemma}[theorem]{Lemma}

\newtheorem{remark}[theorem]{Remark}
\newtheorem{assumption}[theorem]{Assumption}
\numberwithin{equation}{section}

\newcommand{\st}{\textnormal{s.t.}}

\,
\,

\newcommand{\argmin}{\mathop{\rm argmin}}
\newcommand{\LCal}{\mathcal{L}}

\def\blue#1{\textcolor{black}{#1}}
\newcommand{\br}{\mathbb{R}}

\newcommand{\be}{\begin{equation}}
\newcommand{\ee}{\end{equation}}
\newcommand{\ba}{\begin{array}}
\newcommand{\ea}{\end{array}}
\newcommand{\bpm}{\begin{pmatrix}}
\newcommand{\epm}{\end{pmatrix}}

\newcommand{\XCal}{\mathcal{X}}

\newcommand{\etal}{{et al. }}

\begin{document}

\title{On the Sublinear Convergence Rate of Multi-Block ADMM}

\author{Tianyi Lin\thanks{Department of Systems Engineering and Engineering Management, The Chinese University of Hong Kong, Shatin, New Territories, Hong Kong, China. Email: linty@se.cuhk.edu.hk; sqma@se.cuhk.edu.hk. Research of S. Ma was supported in part by the Hong Kong Research Grants Council General Research Fund Early Career Scheme (Project ID: CUHK 439513).}
\and Shiqian Ma\footnotemark[1]
\and Shuzhong Zhang\thanks{Department of Industrial and Systems Engineering, University of Minnesota, Minneapolis, MN 55455, USA. Email: zhangs@umn.edu. Research of S. Zhang was supported in part by the NSF Grant CMMI-1161242.}}

\date{June 29, 2015}

\maketitle

\begin{abstract}

The alternating direction method of multipliers (ADMM) is widely used in solving structured convex optimization problems. Despite of its success in practice, the convergence of the standard ADMM for minimizing the sum of $N$ $(N\geq 3)$ convex functions whose variables are linked by linear constraints, has remained unclear for a very long time. Recently, Chen \etal \cite{Chen-admm-failure-2013} provided a counter-example showing that the ADMM for $N\geq 3$ may fail to converge without further conditions. Since the ADMM for $N\geq 3$ has been very successful when applied to many problems arising from real practice, it is worth further investigating under what kind of sufficient conditions it can be guaranteed to converge. In this paper, we present such sufficient conditions that can guarantee the sublinear convergence rate for the ADMM for $N\geq 3$. Specifically, we show that if one of the functions is convex (not necessarily strongly convex) and the other $N-1$ functions are strongly convex, and the penalty parameter lies in a certain region, the ADMM converges with rate $O(1/t)$ in a certain ergodic sense, and $o(1/t)$ in a certain non-ergodic sense, where $t$ denotes the number of iterations. \blue{As a by-product, we also provide a simple proof for the $O(1/t)$ convergence rate of two-block ADMM in terms of both objective error and constraint violation, without assuming any condition on the penalty parameter and strong convexity on the functions.}

\vspace{0.8cm}




\end{abstract}


\section{Introduction}

We consider solving the following multi-block convex minimization problem:
\be\label{prob:N}\ba{ll} \min & f_1(x_1) + f_2(x_2) + \cdots + f_N(x_N) \\
                         \st  & A_1 x_1  + A_2 x_2 + \cdots + A_N x_N = b \\
                              & x_i \in \mathcal{X}_i, \, i = 1,\ldots, N, \ea \ee
where $A_i \in \br^{p\times n_i}$, $b\in\br^p$, $\XCal_i\subset \br^{n_i}$ are closed convex sets, and $f_i:\br^{n_i}\rightarrow\br^p$ are closed convex functions. One recently popular way to solve \eqref{prob:N}, when the functions $f_i$'s are of special structures, is to apply the alternating direction method of multipliers (ADMM) \cite{Glowinski-Marrocco-1975,Gabay-Mercier-1976}. The ADMM is closely related to the Douglas-Rachford \cite{Douglas-Rachford-56} and Peaceman-Rachford \cite{Peaceman-Rachford-55} operator splitting methods that date back to 1950s. These operator splitting methods were further studied later in \cite{Lions-Mercier-79,Fortin-Glowinski-1983,Glowinski-LeTallec-89,Eckstein-thesis-89}.
The ADMM has been revisited recently due to its success in solving problems with special structures arising from compressed sensing, machine learning, image processing, and so on; see the recent survey papers \cite{Boyd-etal-ADM-survey-2011,Eckstein-tutorial-admm} for more information.

ADMM for solving \eqref{prob:N} is based on an augmented Lagrangian method framework. The augmented Lagrangian function for \eqref{prob:N} is defined as
\[ \LCal_\gamma(x_1,\ldots,x_N;\lambda) := \sum_{j=1}^N f_j(x_j) - \left\langle \lambda, \sum_{j=1}^N A_j x_j -b\right\rangle + \frac{\gamma}{2}\left\|\sum_{j=1}^N A_j x_j -b\right\|^2,\]
where $\lambda$ is the Lagrange multiplier and $\gamma > 0$ is a penalty parameter.
In a typical iteration of the standard ADMM for solving \eqref{prob:N}, the following updating procedure is implemented:
\be\label{admm-N}
\left\{\ba{lcl} x_1^{k+1} & := & \argmin_{x_1\in \XCal_1} \ \LCal_\gamma(x_1,x_2^k,\ldots,x_N^k;\lambda^k) \\
                x_2^{k+1} & := & \argmin_{x_2\in \XCal_2} \ \LCal_\gamma(x_1^{k+1},x_2,x_3^k,\ldots,x_N^k;\lambda^k) \\
                          & \vdots & \\
                x_N^{k+1} & := & \argmin_{x_N\in \XCal_N} \ \LCal_\gamma(x_1^{k+1},x_2^{k+1},\ldots,x_{N-1}^{k+1},x_N;\lambda^k) \\
                \lambda^{k+1} & := & \lambda^k - \gamma \left(\sum_{j=1}^N A_j x_j^{k+1} -b\right).           \ea\right. \ee

The ADMM \eqref{admm-N} for solving two-block convex minimization problems (i.e., $N=2$) has been studied extensively in the literature.
The global convergence of ADMM \eqref{admm-N} when $N=2$ has been shown in \cite{Gabay-83,Eckstein-Bertsekas-1992}. There are also some very recent works that study the convergence rate properties of ADMM when $N=2$ (see, e.g., \cite{He-Yuan-rate-ADM-2012,Monteiro-Svaiter-2010a,Deng-Yin-2012,Boley-2012,He-Yuan-nonergodic-2012,Davis-Yin-2014-Regularity-admm}).

However, the convergence of ADMM \eqref{admm-N} when $N\geq 3$ had remained unclear for a long time.
In a recent work by Chen \etal \cite{Chen-admm-failure-2013}, a counter-example was constructed that shows the failure of ADMM \eqref{admm-N} when $N\geq 3$. Since the ADMM \eqref{admm-N} for $N\geq 3$ has been successfully applied to solve many problems arising from real practice (see e.g., \cite{Tao-Yuan-SPCP-2011,Wright-RASL-TPAMI}), it is worth investigating under what kind of sufficient conditions the ADMM \eqref{admm-N} can converge. Moreover, it has been observed by many researchers that the ADMM \eqref{admm-N} often outperforms all its modified versions (see the observations in \cite{Wang-etal-2013-multi-2,sun-toh-yang-admm-2014}). In fact, Sun, Toh and Yang made the following statement in \cite{sun-toh-yang-admm-2014}:
    ``However, to the best of our knowledge, up to now the dilemma is that at least
    for convex conic programming, the modified versions though with convergence guarantee, often
    perform 2-3 times slower than the multi-block ADMM with no convergent guarantee.''
There is thus a strong need to further study sufficient conditions that can guarantee the convergence of \eqref{admm-N}.
It was shown by Han and Yuan in \cite{Han-Yuan-note-2012} that ADMM \eqref{admm-N} globally converges if all the functions $f_1,\ldots,f_N$ are assumed to be strongly convex and the penalty parameter $\gamma$ is smaller than a certain bound.  Chen, Shen and You \cite{chen-shen-you-admm3} showed that
the 3-block ADMM (i.e., $N=3$ in \eqref{admm-N}) globally converges if $A_1$ is injective, $f_2$ and $f_3$ and strongly convex and $\gamma$ is smaller than a certain bound. \blue{After we released our work\footnote{Preprint available at http://arxiv.org/abs/1408.4265}, Cai, Han and Yuan \cite{cai-han-yuan-2014} and Li, Sun and Toh \cite{Li-Sun-Toh-2014-admm} independently proved that when $N=3$, the ADMM \eqref{admm-N} converges under the conditions that one function among $f_1$, $f_2$ and $f_3$ is strongly convex and $\gamma$ is smaller than a certain bound. Davis and Yin \cite{Davis-Yin-2015} studied a variant of the 3-block ADMM (see Algorithm 8 in \cite{Davis-Yin-2015}) which requires that $f_1$ is strongly convex and $\gamma$ is smaller than a certain bound to guarantee the convergence.  
Recently, Lin, Ma and Zhang \cite{Lin-Ma-Zhang-2015} proposed several alternative approaches to ensure the sublinear convergence rate of \eqref{admm-N} without requiring any function to be strongly convex. Furthermore, Lin, Ma and Zhang \cite{Lin-Ma-Zhang-2015-free-gamma} proved that the 3-block ADMM is globally convergent for any $\gamma>0$ when it is applied to solve the so-called regularized least squares decomposition problems.}
In a recent work by Hong and Luo \cite{Luo-ADMM-2012}, a variant of ADMM \eqref{admm-N} with small step size in updating the Lagrange multiplier was studied. Specifically, \cite{Luo-ADMM-2012} proposed to replace the last equation in \eqref{admm-N} by
\[\lambda^{k+1} := \lambda^k - \alpha \gamma \left(\sum_{j=1}^N A_j x_j^{k+1} -b\right),\]
where $\alpha>0$ is a small step size. Linear convergence of this variant is proved under the assumption that the objective function satisfies certain error bound conditions. However, it is noted that the selection of $\alpha$ is in fact bounded by some parameters associated with the error bound conditions to guarantee the convergence. Therefore, it might be difficult to choose $\alpha$ in practice.
There are also studies on the convergence rate of some other variants of ADMM \eqref{admm-N}, and we refer the interested readers to  \cite{He-Tao-Yuan-2012,He-Tao-Yuan-MOR-2013,He-Hou-Yuan-Jacob-2013,Deng-admm-2014,Hong-etal-2014-BSUMM} for details of these variants. In this paper, we focus on the {\it ADMM \eqref{admm-N} that directly extends the two-block ADMM to problems with more than two block variables}.

{\bf Our contributions.} The main contribution in this paper are as follows. We show that the ADMM \eqref{admm-N} when $N\geq 3$ converges with rate $O(1/t)$ in ergodic sense and $o(1/t)$ in non-ergodic sense, under the assumption that $f_2,\ldots,f_N$ are strongly convex and $f_1$ is convex but not necessarily strongly convex, and $\gamma$ \blue{is smaller than a certain bound}. It should be pointed out that our assumption is weaker than the one used in \cite{Han-Yuan-note-2012}, in which all the functions are required to be strongly convex. Moreover, unlike the sufficient condition suggested in \cite{Chen-admm-failure-2013}, we do not make any assumption on the matrices $A_1,\ldots,A_N$. To the best of our knowledge, the convergence rate results given in this paper are the first sublinear convergence rate results for the standard ADMM \eqref{admm-N} when $N\geq 3$. We also remark here that by further assuming additional conditions, we proved the global linear convergence rate of ADMM \eqref{admm-N} in \cite{Lin-Ma-Zhang-2014-linear}.

{\bf Organization.} The rest of this paper is organized as follows. In Section \ref{sec:pre} we provide some preliminaries for our convergence rate analysis. In Section \ref{sec:ergodic}, we prove the convergence rate of ADMM \eqref{admm-N} in the ergodic sense. In Section \ref{sec:non-ergodic}, we prove the convergence rate of ADMM \eqref{admm-N} in the non-ergodic sense. Section \ref{sec:conclusion} draws some conclusions and points out some future directions.

\section{Preliminaries}\label{sec:pre}

We will only prove the convergence results of ADMM for $N=3$, because all the analysis can be extended to arbitrary $N$ easily. As a result, for the ease of presentation and succinctness, we assume $N=3$ in the rest of this paper. We will present the results for general $N$ but omit the proofs.

We restate the problem \eqref{prob:N} for $N=3$ as
\be\label{prob:P}\ba{ll} \min & f_1(x_1) + f_2(x_2) + f_3(x_3) \\
                         \st  & A_1 x_1 + A_2 x_2 + A_3 x_3 = b \\
                              & x_1\in\XCal_1, x_2\in\XCal_2, x_3\in\XCal_3. \ea \ee
The ADMM for solving \eqref{prob:P} can be summarized as (note that some constant terms in the three subproblems are discarded):
\begin{eqnarray}
x_{1}^{k+1} & := & \argmin_{x_{1} \in \XCal_{1}}f_{1}(x_{1})+\frac{\gamma}{2}\|A_{1}x_{1}+A_{2}x_{2}^{k}+A_{3}x_{3}^{k}-b-\frac{1}{\gamma}\lambda^{k}\|^{2} \label{eq::update_x}\\
x_{2}^{k+1} & := & \argmin_{x_{2} \in \XCal_{2}}f_{2}(x_{2})+\frac{\gamma}{2}\|A_{1}x_{1}^{k+1}+A_{2}x_{2}+A_{3}x_{3}^{k}-b-\frac{1}{\gamma}\lambda^{k}\|^{2} \label{eq::update_y} \\
x_{3}^{k+1} & := & \argmin_{x_{3} \in \XCal_{3}}f_{3}(x_{3})+\frac{\gamma}{2}\|A_{1}x_{1}^{k+1}+A_{2}x_{2}^{k+1}+A_{3}x_{3}-b-\frac{1}{\gamma}\lambda^{k}\|^{2}\label{eq::update_z} \\
\lambda^{k+1} &:=& \lambda^{k} - \gamma \left(A_{1}x_{1}^{k+1}+A_{2}x_{2}^{k+1}+A_{3}x_{3}^{k+1}-b\right). \label{eq::update_lambda}
\end{eqnarray}
The first-order optimality conditions for \eqref{eq::update_x}-\eqref{eq::update_z} are given respectively by $x_{i}^{k+1}\in\XCal_{i}, i=1,2,3$, and
\begin{align}
& (x_{1}-x_{1}^{k+1})^{\top}\left[ 
g_1(x_{1}^{k+1})-A_{1}^{\top}\lambda^{k}+\gamma A_{1}^{\top}(A_{1}x_{1}^{k+1}+A_{2}x_{2}^{k}+A_{3}x_{3}^{k}-b)\right]\geq 0, & \quad \forall x_{1}\in\XCal_{1}, \label{opt-x1} \\
& (x_{2}-x_{2}^{k+1})^{\top}\left[ 
g_2(x_{2}^{k+1})-A_{2}^{\top}\lambda^{k}+\gamma A_{2}^{\top}(A_{1}x_{1}^{k+1}+A_{2}x_{2}^{k+1}+A_{3}x_{3}^{k}-b)\right] \geq 0, &  \quad \forall x_{2}\in\XCal_{2}, \label{opt-x2} \\
& (x_{3}-x_{3}^{k+1})^{\top}\left[ 
g_3(x_{3}^{k+1})-A_{3}^{\top}\lambda^{k}+\gamma A_{3}^{\top}(A_{1}x_{1}^{k+1}+A_{2}x_{2}^{k+1}+A_{3}x_{3}^{k+1}-b)\right] \geq 0, &  \quad \forall x_{3}\in\XCal_{3}, \label{opt-x3}
\end{align}
where $g_i \in \partial f_i$ is the subgradient of $f_i$ for $i=1,2,3$.
Moreover, by combining with \eqref{eq::update_lambda}, \eqref{opt-x1}-\eqref{opt-x3} can be rewritten as
\begin{align}
& (x_{1}-x_{1}^{k+1})^{\top}\left[ g_{1}(x_{1}^{k+1})-A_{1}^{\top}\lambda^{k+1} 
+\gamma A_{1}^{\top}A_{2}(x_{2}^{k}-x_{2}^{k+1}) + \gamma A_{1}^{\top}A_{3}(x_{3}^{k}-x_{3}^{k+1})\right] \geq 0, & \forall x_{1}\in\XCal_{1}, \label{opt-x1-lambda} \\
& (x_{2}-x_{2}^{k+1})^{\top}\left[ g_2(x_{2}^{k+1})-A_{2}^{\top}\lambda^{k+1}+\gamma A_{2}^{\top}A_{3}(x_{3}^{k}-x_{3}^{k+1})\right] \geq 0, & \forall x_{2}\in\XCal_{2}, \label{opt-x2-lambda} \\
& (x_{3}-x_{3}^{k+1})^{\top}\left[ g_{3}(x_{3}^{k+1})-A_{3}^{\top}\lambda^{k+1}\right] \geq 0, & \forall x_{3}\in\XCal_{3}. \label{opt-x3-lambda}
\end{align}

We denote $\Omega = \XCal_1\times \XCal_2 \times \XCal_3 \times \br^p$ and the optimal set of \eqref{prob:P} as $\Omega^*$, and the following assumption is made throughout this paper.
\begin{assumption} The optimal set $\Omega^*$ for problem \eqref{prob:P} is non-empty.\end{assumption}
According to the first-order optimality conditions for \eqref{prob:P}, solving \eqref{prob:P} is equivalent to finding $$(x_1^*,x_2^*,x_3^*,\lambda^*)\in\Omega^*$$
such that the following holds:
\begin{equation}\label{kkt}
\left\{
\begin{array}{l}
(x_1-x_1^*)^\top (g_1(x_1^*)-A_1^\top\lambda^*) \geq 0, \forall x_1\in\XCal_1, \\
(x_2-x_2^*)^\top (g_2(x_2^*)-A_2^\top\lambda^*) \geq 0, \forall x_2\in\XCal_2, \\
(x_3-x_3^*)^\top (g_3(x_3^*)-A_3^\top\lambda^*) \geq 0, \forall x_3\in\XCal_3, \\
A_{1}x_{1}^{*}+A_{2}x_{2}^{*}+A_{3}x_{3}^{*}-b=0,
\end{array}
\right.
\end{equation}
where $g_i(x_i^*)\in\partial f_i(x_i^*)$, $i=1,2,3$.

Furthermore, the following condition is assumed in our subsequent analysis.
\begin{assumption}\label{assump:strongly-convex}
The functions $f_2$ and $f_3$ are strongly convex with parameters $\sigma_2>0$ and $\sigma_3>0$, respectively; i.e., the following two inequalities hold:
\begin{eqnarray}
f_2(y) &\geq& f_2(x) + (y-x)^\top g_2(x) + \frac{\sigma_2}{2}\|y-x\|^2, \quad \forall x,y\in\XCal_2, \label{f2-strong-1} \\
f_3(y) &\geq& f_3(x) + (y-x)^\top g_3(x) + \frac{\sigma_3}{2}\|y-x\|^2, \quad \forall x,y\in\XCal_3, \label{f3-strong-1}
\end{eqnarray}
or equivalently,
\begin{eqnarray}
(y-x)^\top (g_2(y)-g_2(x)) &\geq& \sigma_2 \|y-x\|^2, \quad \forall x,y\in\XCal_2, \label{f2-strong-2} \\
(y-x)^\top (g_3(y)-g_3(x)) &\geq& \sigma_3 \|y-x\|^2, \quad \forall x,y\in\XCal_3,  \label{f3-strong-2}
\end{eqnarray}
where $g_2(x)\in \partial f_2(x)$ and $g_3(x) \in \partial f_3(x)$ are the subgradients of $f_2$ and $f_3$ respectively.
\end{assumption}

\blue{
In our analysis, the following well-known identity is used frequently,
\be\label{identity-4}
 (w_1-w_2)^{\top}(w_3-w_4)= \frac{1}{2}\left(\|w_1-w_4\|^{2}-\|w_1-w_3\|^{2}\right)+\frac{1}{2}\left(\|w_3-w_2\|^{2}-\|w_4-w_2\|^{2}\right).
\ee
}

{\bf Notations. }
For simplicity, we use the following notation to denote the stacked vectors or tuples:
\[ u = \left(\begin{array}{c} x_{1} \\ x_{2} \\ x_{3} \end{array} \right),
u^k = \left(\begin{array}{c} x_{1}^k \\ x_{2}^k \\ x_{3}^k \end{array} \right),
u^* = \left(\begin{array}{c} x_{1}^* \\ x_{2}^* \\ x_{3}^* \end{array} \right).\]
We denote by $f(u)\equiv f_1(x_{1})+f_{2}(x_{2})+f_{3}(x_{3})$ the objective function of problem \eqref{prob:P}; $g_i$ is a subgradient of $f_i$;
$\lambda_{\max}(B)$ denotes the largest eigenvalue of a real symmetric matrix $B$; $\|x\|$ denotes the Euclidean norm of $x$.

\section{Ergodic Convergence Rate of ADMM}\label{sec:ergodic}

In this section, we prove the $O(1/t)$ convergence rate of ADMM \eqref{eq::update_x}-\eqref{eq::update_lambda} in the ergodic sense.

\begin{lemma}\label{lemma1}
Assume that $\gamma \leq \min\left\{\frac{\sigma_2}{2\lambda_{\max}(A_{2}^{\top}A_{2})}, \frac{\sigma_3}{2\lambda_{\max}(A_{3}^{\top}A_{3})}\right\}$, where $\sigma_2$ and $\sigma_3$ are defined in Assumption \ref{assump:strongly-convex}. Let $(x_1^{k+1},x_2^{k+1},x_3^{k+1},\lambda^{k+1})\in\Omega$ be generated by ADMM from given $(x_{2}^{k},x_{3}^{k},\lambda^{k})$. \blue{Then, for any primal optimal solution $u^*=(x_1^*,x_2^*,x_3^*)$ of \eqref{prob:P} and $\lambda\in\br^p$, it holds that}
\begin{eqnarray}
& & f(u^{*})-f(u^{k+1})+\left(\begin{array}{c} x_{1}^{*}-x_{1}^{k+1} \\ x_{2}^{*}-x_{2}^{k+1} \\ x_{3}^{*}-x_{3}^{k+1} \\ \lambda-\lambda^{k+1}\end{array} \right)^{\top}
\left(\begin{array}{c} -A_{1}^{\top}\lambda^{k+1} \\ -A_{2}^{\top}\lambda^{k+1} \\ -A_{3}^{\top}\lambda^{k+1} \\ A_{1}x_{1}^{k+1}+A_{2}x_{2}^{k+1}+A_{3}x_{3}^{k+1}-b \end{array} \right) \nonumber \\
& & + \frac{1}{2\gamma}\left(\|\lambda-\lambda^{k}\|^{2}-\|\lambda-\lambda^{k+1}\|^{2}\right)+\frac{\gamma}{2}\left(\| A_{1}x_{1}^{*}+A_{2}x_{2}^{*}+A_{3}x_{3}^{k}-b\|^{2}-\| A_{1}x_{1}^{*}+A_{2}x_{2}^{*}+A_{3}x_{3}^{k+1}-b\|^{2}\right) \nonumber \\
& & +\frac{\gamma}{2}\left(\| A_{1}x_{1}^{*}+A_{2}x_{2}^{k}+A_{3}x_{3}^{k}-b\|^{2}-\| A_{1}x_{1}^{*}+A_{2}x_{2}^{k+1}+A_{3}x_{3}^{k+1}-b\|^{2}\right) \nonumber \\
&\geq&  \frac{\gamma}{2}\| A_{1}x_{1}^{k+1}+A_{2}x_{2}^{k}+A_{3}x_{3}^{k}-b\|^{2}. \label{1}
\end{eqnarray}
\end{lemma}

\begin{proof}
Note that combining \eqref{opt-x1-lambda}-\eqref{opt-x3-lambda} yields
\begin{eqnarray}\label{1_3}
\begin{aligned}
& \left(\begin{array}{c} x_{1}-x_{1}^{k+1} \\ x_{2}-x_{2}^{k+1} \\ x_{3}-x_{3}^{k+1} \end{array} \right)^{\top}\left[
\left(\begin{array}{c} g_1 
(x_{1}^{k+1})-A_{1}^{\top}\lambda^{k+1} \\ 
g_{2}(x_{2}^{k+1})-A_{2}^{\top}\lambda^{k+1} \\ 
g_{3}(x_{3}^{k+1})-A_{3}^{\top}\lambda^{k+1} \end{array} \right) + \left(\begin{array}{ccc} \gamma A_{1}^{\top}A_{2} & \gamma A_{1}^{\top}A_{3} \\ 0 & \gamma A_{2}^{\top}A_{3} \\ 0 & 0 \end{array} \right)
\left(\begin{array}{c} x_{2}^{k}-x_{2}^{k+1} \\ x_{3}^{k}-x_{3}^{k+1} \end{array} \right)\right]\geq 0.
\end{aligned}
\end{eqnarray}
The key step in our proof is to bound the following two terms
\begin{displaymath}
(x_{1}-x_{1}^{k+1})^{\top}A_{1}^{\top}(A_{2}(x_{2}^{k}-x_{2}^{k+1})+A_{3}(x_{3}^{k}-x_{3}^{k+1})) \quad \mbox{ and } \quad (x_{2}-x_{2}^{k+1})^{\top}A_{2}^{\top}A_{3}(x_{3}^{k}-x_{3}^{k+1}).
\end{displaymath}

For the first term, we have
\begin{eqnarray*}
& & (x_{1}-x_{1}^{k+1})^{\top}A_{1}^{\top}\left[A_{2}(x_{2}^{k}-x_{2}^{k+1})+A_{3}(x_{3}^{k}-x_{3}^{k+1})\right] \\
& = & \left[(A_{1}x_{1}-b)-(A_{1}x_{1}^{k+1}-b)\right]^{\top}\left[(-A_{2}x_{2}^{k+1}-A_{3}x_{3}^{k+1})-(-A_{2}x_{2}^{k}-A_{3}x_{3}^{k})\right] \\
& = & \frac{1}{2}\left(\| A_{1}x_{1}+A_{2}x_{2}^{k}+A_{3}x_{3}^{k}-b\|^{2}-\| A_{1}x_{1}+A_{2}x_{2}^{k+1}+A_{3}x_{3}^{k+1}-b\|^{2}\right)\\
&  & +\frac{1}{2}\left(\| A_{1}x_{1}^{k+1}+A_{2}x_{2}^{k+1}+A_{3}x_{3}^{k+1}-b\|^{2} -\| A_{1}x_{1}^{k+1}+A_{2}x_{2}^{k}+A_{3}x_{3}^{k}-b\|^{2}\right)\\
& = & \frac{1}{2}\left(\| A_{1}x_{1}+A_{2}x_{2}^{k}+A_{3}x_{3}^{k}-b\|^{2}-\| A_{1}x_{1}+A_{2}x_{2}^{k+1}+A_{3}x_{3}^{k+1}-b\|^{2}\right) + \frac{1}{2\gamma^{2}}\|\lambda^{k+1}-\lambda^{k}\|^{2}\\
& & -\frac{1}{2}\| A_{1}x_{1}^{k+1}+A_{2}x_{2}^{k}+A_{3}x_{3}^{k}-b\|^{2},
\end{eqnarray*}
where in the second equality we used the identity \eqref{identity-4}, and the last equality follows from the updating formula for $\lambda^{k+1}$ in \eqref{eq::update_lambda}.

For the second term, we have
\begin{eqnarray*}
& & (x_{2}-x_{2}^{k+1})^{\top}A_{2}^{\top}A_{3}(x_{3}^{k}-x_{3}^{k+1}) \\
& = & ((A_{1}x_{1}+A_{2}x_{2}-b)-(A_{1}x_{1}+A_{2}x_{2}^{k+1}-b))^{\top}((-A_{3}x_{3}^{k+1})-(-A_{3}x_{3}^{k}))\\
& = & \frac{1}{2}\left(\| A_{1}x_{1}+A_{2}x_{2}+A_{3}x_{3}^{k}-b\|^{2}-\| A_{1}x_{1}+A_{2}x_{2}+A_{3}x_{3}^{k+1}-b\|^{2}\right)\\
&   & +\frac{1}{2}\left(\| A_{1}x_{1}+A_{2}x_{2}^{k+1}+A_{3}x_{3}^{k+1}-b\|^{2} -\| A_{1}x_{1}+A_{2}x_{2}^{k+1}+A_{3}x_{3}^{k}-b\|^{2}\right)\\
&\leq & \frac{1}{2}\left(\| A_{1}x_{1}+A_{2}x_{2}+A_{3}x_{3}^{k}-b\|^{2}-\| A_{1}x_{1}+A_{2}x_{2}+A_{3}x_{3}^{k+1}-b\|^{2}\right)\\
& & +\frac{1}{2}\| A_{1}x_{1}+A_{2}x_{2}^{k+1}+A_{3}x_{3}^{k+1}-b\|^{2},
\end{eqnarray*}
where in the second equality we applied the identity \eqref{identity-4}.

Therefore, we have
\begin{eqnarray}
& & (x_{1}-x_{1}^{k+1})^{\top}\gamma A_{1}^{\top}(A_{2}(x_{2}^{k}-x_{2}^{k+1})+A_{3}(x_{3}^{k}-x_{3}^{k+1}))+(x_{2}-x_{2}^{k+1})^{\top}\gamma A_{2}^{\top}A_{3}(x_{3}^{k}-x_{3}^{k+1}) \nonumber \\
 &\leq & \frac{\gamma}{2}\left(\| A_{1}x_{1}+A_{2}x_{2}^{k}+A_{3}x_{3}^{k}-b\|^{2}-\| A_{1}x_{1}+A_{2}x_{2}^{k+1}+A_{3}x_{3}^{k+1}-b\|^{2}\right) \nonumber \\
 & & +\frac{\gamma}{2}\left(\| A_{1}x_{1}+A_{2}x_{2}+A_{3}x_{3}^{k}-b\|^{2}-\| A_{1}x_{1}+A_{2}x_{2}+A_{3}x_{3}^{k+1}-b\|^{2}\right) \nonumber \\
 & & + \frac{1}{2\gamma}\|\lambda^{k+1}-\lambda^{k}\|^{2}+\frac{\gamma}{2}\| A_{1}x_{1}+A_{2}x_{2}^{k+1}+A_{3}x_{3}^{k+1}-b\|^{2} \nonumber \\
 & &  -\frac{\gamma}{2}\| A_{1}x_{1}^{k+1}+A_{2}x_{2}^{k}+A_{3}x_{3}^{k}-b\|^{2}. \label{1_3_1}
\end{eqnarray}
Combining \eqref{1_3_1}, \eqref{1_3} and \eqref{eq::update_lambda}, it holds for any $\lambda\in\br^p$ that
\begin{eqnarray}
& & \left(\begin{array}{c} x_{1}-x_{1}^{k+1} \\ x_{2}-x_{2}^{k+1} \\ x_{3}-x_{3}^{k+1} \\ \lambda-\lambda^{k+1}\end{array} \right)^{\top}
\left(\begin{array}{c} 
g_{1}(x_{1}^{k+1})-A_{1}^{\top}\lambda^{k+1} \\ 
g_{2}(x_{2}^{k+1})-A_{2}^{\top}\lambda^{k+1} \\ 
g_{3}(x_{3}^{k+1})-A_{3}^{\top}\lambda^{k+1} \\ A_{1}x_{1}^{k+1}+A_{2}x_{2}^{k+1}+A_{3}x_{3}^{k+1}-b \end{array} \right)+ \frac{1}{\gamma}(\lambda-\lambda^{k+1})^{\top}(\lambda^{k+1}-\lambda^{k}) \nonumber \\
& & +\frac{1}{2\gamma}\|\lambda^{k+1}-\lambda^{k}\|^{2} +\frac{\gamma}{2}\| A_{1}x_{1}+A_{2}x_{2}^{k+1}+A_{3}x_{3}^{k+1}-b\|^{2} \nonumber \\
& & +\frac{\gamma}{2}\left(\| A_{1}x_{1}+A_{2}x_{2}^{k}+A_{3}x_{3}^{k}-b\|^{2}-\| A_{1}x_{1}+A_{2}x_{2}^{k+1}+A_{3}x_{3}^{k+1}-b\|^{2}\right) \nonumber \\
& & +\frac{\gamma}{2}\left(\| A_{1}x_{1}+A_{2}x_{2}+A_{3}x_{3}^{k}-b\|^{2}-\| A_{1}x_{1}+A_{2}x_{2}+A_{3}x_{3}^{k+1}-b\|^{2}\right) \nonumber \\
& \geq & \frac{\gamma}{2}\| A_{1}x_{1}^{k+1}+A_{2}x_{2}^{k}+A_{3}x_{3}^{k}-b\|^{2}. \label{1_4}
\end{eqnarray}
\blue{Using the convexity of $f_1$ and the identity}
\begin{eqnarray*}
\frac{1}{\gamma}\left(\lambda-\lambda^{k+1}\right)^{\top}(\lambda^{k+1}-\lambda^{k})+\frac{1}{2\gamma}\|\lambda^{k+1}-\lambda^{k}\|^{2} = \frac{1}{2\gamma}\left(\|\lambda-\lambda^{k}\|^{2}-\|\lambda-\lambda^{k+1}\|^{2}\right),
\end{eqnarray*}
letting $u=u^{*}$ in \eqref{1_4}, and applying the facts that (invoking \eqref{f2-strong-1} and \eqref{f3-strong-1})
\begin{eqnarray*}\label{1_5_2}
\begin{aligned}
& f_{2}(x_{2}^{*})-f_{2}(x_{2}^{k+1})-\frac{\sigma_2}{2}\| x_{2}^{*}-x_{2}^{k+1}\|^{2} \geq (x_{2}^{*}-x_{2}^{k+1})^{\top}
g_{2}(x_{2}^{k+1}),\\
& f_{3}(x_{3}^{*})-f_{3}(x_{3}^{k+1})-\frac{\sigma_3}{2}\| x_{3}^{*}-x_{3}^{k+1}\|^{2} \geq (x_{3}^{*}-x_{3}^{k+1})^{\top} 
g_{3}(x_{3}^{k+1}),
\end{aligned}
\end{eqnarray*}
and
\begin{eqnarray*}\label{1_5_1}
& & \frac{\gamma}{2}\| A_{1}x_{1}^{*}+A_{2}x_{2}^{k+1}+A_{3}x_{3}^{k+1}-b\|^{2} \\
& = &  \frac{\gamma}{2}\| A_{2}(x_{2}^{k+1}-x_{2}^{*})+A_{3}(x_{3}^{k+1}-x_{3}^{*})\|^{2}\\
& \leq & \gamma(\lambda_{\max}(A_{2}^{\top}A_{2})\| x_{2}^{k+1}-x_{2}^{*}\|^{2}+\lambda_{\max}(A_{3}^{\top}A_{3})\| x_{3}^{k+1}-x_{3}^{*}\|^{2}),
\end{eqnarray*}
we obtain,
\begin{eqnarray*}
& & f(u^{*})-f(u^{k+1})+\left(\begin{array}{c} x_{1}^{*}-x_{1}^{k+1} \\ x_{2}^{*}-x_{2}^{k+1} \\ x_{3}^{*}-x_{3}^{k+1} \\ \lambda-\lambda^{k+1}\end{array} \right)^{\top}
\left(\begin{array}{c} -A_{1}^{\top}\lambda^{k+1} \\ -A_{2}^{\top}\lambda^{k+1} \\ -A_{3}^{\top}\lambda^{k+1} \\ A_{1}x_{1}^{k+1}+A_{2}x_{2}^{k+1}+A_{3}x_{3}^{k+1}-b \end{array} \right)\\
& & + \frac{1}{2\gamma}\left(\|\lambda-\lambda^{k}\|^{2}-\|\lambda-\lambda^{k+1}\|^{2}\right)+\frac{\gamma}{2}\left(\| A_{1}x_{1}^{*}+A_{2}x_{2}^{*}+A_{3}x_{3}^{k}-b\|^{2}-\| A_{1}x_{1}^{*}+A_{2}x_{2}^{*}+A_{3}x_{3}^{k+1}-b\|^{2}\right) \\
& & +\left(\gamma\lambda_{\max}(A_{2}^{\top}A_{2})-\frac{\sigma_2}{2}\right)\| x_{2}^{k+1}-x_{2}^{*}\|^{2}+\left(\gamma\lambda_{\max}(A_{3}^{\top}A_{3})-\frac{\sigma_3}{2}\right)\| x_{3}^{k+1}-x_{3}^{*}\|^{2}\\
& & +\frac{\gamma}{2}\left(\| A_{1}x_{1}^{*}+A_{2}x_{2}^{k}+A_{3}x_{3}^{k}-b\|^{2}-\| A_{1}x_{1}^{*}+A_{2}x_{2}^{k+1}+A_{3}x_{3}^{k+1}-b\|^{2}\right) \\
& \geq & \frac{\gamma}{2}\| A_{1}x_{1}^{k+1}+A_{2}x_{2}^{k}+A_{3}x_{3}^{k}-b\|^{2}.
\end{eqnarray*}
This together with the facts that $\gamma\lambda_{\max}(A_{2}^{\top}A_{2})-\frac{\sigma_2}{2} \leq 0$ and $\gamma\lambda_{\max}(A_{3}^{\top}A_{3})-\frac{\sigma_3}{2}\leq 0$ implies the desired inequality \eqref{1}.
\end{proof}

Now, we are ready to present the $O(1/t)$ ergodic convergence rate of the ADMM.

\begin{theorem}\label{thm-ergodic-3}
Assume that $\gamma \leq \min\left\{\frac{\sigma_2}{2\lambda_{\max}(A_{2}^{\top}A_{2})}, \frac{\sigma_3}{2\lambda_{\max}(A_{3}^{\top}A_{3})}\right\}$. Let $(x_1^{k+1},x_2^{k+1},x_3^{k+1},\lambda^{k+1})\in\Omega$ be generated by ADMM \eqref{eq::update_x}-\eqref{eq::update_lambda} from given $(x_2^{k},x_3^{k},\lambda^{k})$. For any integer $t>0$, let $\bar{u}^{t}=(\bar{x}_{1}^{t}, \bar{x}_{2}^{t}, \bar{x}_{3}^{t})$ and $\bar{\lambda}^{t}$ be defined as
\begin{eqnarray*}
\bar{x}_{1}^{t}=\frac{1}{t+1}\sum\limits_{k=0}^{t}x_{1}^{k+1},\quad
\bar{x}_{2}^{t}=\frac{1}{t+1}\sum\limits_{k=0}^{t}x_{2}^{k+1},\quad
\bar{x}_{3}^{t}=\frac{1}{t+1}\sum\limits_{k=0}^{t}x_{3}^{k+1},\quad
\bar{\lambda}^{t}=\frac{1}{t+1}\sum\limits_{k=0}^{t}\lambda^{k+1}.
\end{eqnarray*}
Then, for any $(u^*,\lambda^*) \in\Omega^{*}$, by defining $\rho:=\|\lambda^*\|+1$, we have
\begin{eqnarray*}
0 & \leq & f(\bar{u}^{t})-f(u^{*})+\rho\| A_{1}\bar{x}_{1}^{t}+A_{2}\bar{x}_{1}^{t}+A_{3}\bar{x}_{3}^{t}-b\| \\
& \leq & \frac{\gamma}{2(t+1)}\| A_3 x_{3}^{*}-A_3 x_{3}^{0}\|^{2}+\frac{\rho^{2}+\|\lambda^0\|^2}{\gamma (t+1)}+\frac{\gamma}{2(t+1)}\| A_{1}x_{1}^{*}+A_{2}x_{2}^{0}+A_{3}x_{3}^{0}-b\|^{2}.
\end{eqnarray*}
Note that this also implies that both the error of the objective function value and the residual of the equality constraint converge to $0$ with convergence rate $O(1/t)$, i.e.,
\be\label{ergodic-3}|f(\bar{u}^t)-f(u^*)| = O(1/t), \quad \mbox{ and } \quad \| A_{1}\bar{x}_{1}^{t}+A_{2}\bar{x}_{1}^{t}+A_{3}\bar{x}_{3}^{t}-b\|=O(1/t).\ee
\end{theorem}

\begin{proof}
Because $(u^{k},\lambda^k)\in\Omega$, it holds that $(\bar{u}^{t},\bar{\lambda}^t)\in\Omega$ for all $t\geq 0$.
\blue{By Lemma \ref{lemma1}, the last equation of \eqref{kkt}, and invoking the convexity of function $f(\cdot)$, we have}
\begin{eqnarray}\label{new-added}
&   & f(u^*)-f(\bar{u}^t) + \lambda^\top (A_1\bar{x}_1^t+A_2\bar{x}_2^t+A_3\bar{x}_3^t-b) \\
& = & f(u^{*})-f(\bar{u}^{t})+\left(\begin{array}{c} x_{1}^{*}-\bar{x}_{1}^{t} \\ x_{2}^{*}-\bar{x}_{2}^{t} \\ x_{3}^{*}-\bar{x}_{3}^{t} \\ \lambda-\bar{\lambda}^{t}\end{array} \right)^{\top}
\left(\begin{array}{c} -A_{1}^{\top}\bar{\lambda}^{t} \\ -A_{2}^{\top}\bar{\lambda}^{t} \\ -A_{3}^{\top}\bar{\lambda}^{t} \\ A_{1}\bar{x}_{1}^{t}+A_{2}\bar{x}_{2}^{t}+A_{3}\bar{x}_{3}^{t}-b \end{array} \right) \nonumber \\
& \geq & \frac{1}{t+1}\sum\limits_{k=0}^{t}\left[f(u^{*})-f(u^{k+1})+\left(\begin{array}{c} x_{1}^{*}-x_{1}^{k+1} \\ x_{2}^{*}-x_{2}^{k+1} \\ x_{3}^{*}-x_{3}^{k+1} \\ \lambda-\lambda^{k+1}\end{array} \right)^{\top}
\left(\begin{array}{c} -A_{1}^{\top}\lambda^{k+1} \\ -A_{2}^{\top}\lambda^{k+1} \\ -A_{3}^{\top}\lambda^{k+1} \\ A_{1}x_{1}^{k+1}+A_{2}x_{2}^{k+1}+A_{3}x_{3}^{k+1}-b \end{array} \right)\right]\nonumber\\
&\geq &\frac{1}{t+1}\sum\limits_{k=0}^{t}\left[\frac{1}{2\gamma}(\|\lambda-\lambda^{k+1}\|^{2}-\|\lambda-\lambda^{k}\|^{2}) \right.\nonumber \\
& & +\frac{\gamma}{2}\left(\| A_{1}x_{1}^{*}+A_{2}x_{2}^{*}+A_{3}x_{3}^{k+1}-b\|^{2}-\| A_{1}x_{1}^{*}+A_{2}x_{2}^{*}+A_{3}x_{3}^{k}-b\|^{2}\right) \nonumber \\
& & \left. +\frac{\gamma}{2}\left(\| A_{1}x_{1}^{*}+A_{2}x_{2}^{k+1}+A_{3}x_{3}^{k+1}-b\|^{2}-\| A_{1}x_{1}^{*}+A_{2}x_{2}^{k}+A_{3}x_{3}^{k}-b\|^{2}\right)\right]\nonumber\\
& \geq & -\frac{1}{2\gamma(t+1)}\|\lambda-\lambda^{0}\|^{2}-\frac{\gamma}{2(t+1)}\| A_{1}x_{1}^{*}+A_{2}x_{2}^{*}+A_{3}x_{3}^{0}-b\|^{2}-\frac{\gamma}{2(t+1)}\| A_{1}x_{1}^{*}+A_{2}x_{2}^{0}+A_{3}x_{3}^{0}-b\|^{2}.\nonumber
\end{eqnarray}
Note that this inequality holds for all $\lambda\in\br^{p}$. 
From weak duality of \eqref{prob:P} we obtain
\[0 \geq f(u^*)-f(\bar{u}^t)+(\lambda^*)^\top(A_{1}\bar{x}_{1}^{t}+A_{2}\bar{x}_{1}^{t}+A_{3}\bar{x}_{3}^{t}-b),\]
which implies that
\be \label{ergodic-3-inequality-1-1} 0 \leq f(\bar{u}^{t})-f(u^{*})+\rho\| A_{1}\bar{x}_{1}^{t}+A_{2}\bar{x}_{1}^{t} +A_{3}\bar{x}_{3}^{t}-b\|,\ee
because $\rho=\|\lambda^*\|+1$. Moreover,
by letting $\lambda:=-\rho(A_1\bar{x}_1^t+A_2\bar{x}_2^t+A_3\bar{x}_3^t-b)/\|A_1\bar{x}_1^t+A_2\bar{x}_2^t+A_3\bar{x}_3^t-b\|_2$ in \eqref{new-added}, and using $A_1x^*_1+A_2x^*_2+A_3x^*_3=b$, we obtain
\begin{align}\label{ergodic-3-inequality-1}
& f(\bar{u}^{t})-f(u^{*})+\rho\| A_{1}\bar{x}_{1}^{t}+A_{2}\bar{x}_{1}^{t} +A_{3}\bar{x}_{3}^{t}-b\| \nonumber \\
\leq & \frac{\rho^{2}+\|\lambda^0\|^2}{\gamma (t+1)} +\frac{\gamma}{2(t+1)}\| A_3( x_{3}^{*}-x_{3}^{0})\|^{2}+\frac{\gamma}{2(t+1)}\| A_{2}(x_{2}^{*}-x_2^0)+A_3(x^*_3-x_{3}^{0})\|^{2}.
\end{align}

We now define the function
\[v(\xi) = \min \{ f(u) | A_1 x_1 + A_2 x_2 + A_3 x_3 - b = \xi, x_1\in \XCal_1, x_2\in \XCal_2, x_3\in \XCal_3 \}.\]
It is easy to verify that $v$ is convex, $v(0)=f(u^*)$, and $\lambda^* \in \partial v(0)$.
Therefore, from the convexity of $v$, it holds that
\be\label{v-convex} v(\xi) \ge v(0) + \langle \lambda^*, \xi \rangle \ge f(u^*) - \| \lambda^*\| \|\xi\|.\ee
Let $\bar{\xi} = A_1 \bar{x}_1 + A_2 \bar{x}_2 + A_3\bar{x}_3 - b$, we have $f(\bar{u}^t) \ge v(\bar{\xi})$.
Therefore, by denoting the constant
\[C:=\frac{\gamma}{2}\|A_3x_3^*-A_3x_3^0\|^2 + \frac{\|\lambda^0\|^2}{\gamma} + \frac{\gamma}{2}\|A_1x_1^*+A_2x_2^0+A_3x_3^0-b\|^2,\]
and combining \eqref{ergodic-3-inequality-1-1}, \eqref{ergodic-3-inequality-1} and \eqref{v-convex}, we get
\[\frac{C+\rho^2/\gamma}{t+1} - \rho \| \bar{\xi}\| \geq f(\bar{u}^t) - f(u^*) \geq - \| \lambda^*\| \| \bar{\xi} \|,\]
which, by using $\rho=\|\lambda^*\|+1$, yields,
\be\label{infeasibility} \|A_1 \bar{x}_1 + A_2 \bar{x}_2 + A_3\bar{x}_3 - b\| = \|\bar{\xi}\| \leq \frac{C+\rho^2/\gamma}{t+1}. \ee
Moreover, by combining \eqref{ergodic-3-inequality-1-1}, \eqref{ergodic-3-inequality-1} and \eqref{infeasibility}, one obtains that
\be\label{obj-diff} -\frac{\rho C+\rho^3/\gamma}{t+1} \leq f(\bar{u}^t) - f(u^*) \leq \frac{C+\rho^2/\gamma}{t+1}. \ee
As a result, \eqref{ergodic-3} follows immediately from \eqref{infeasibility} and \eqref{obj-diff}.
\end{proof}

Therefore, we have established the $O(1/t)$ convergence rate of the ADMM \eqref{eq::update_x}-\eqref{eq::update_lambda} in an ergodic sense. Our proof is readily extended to the case of $N$-block ADMM \eqref{admm-N}. The following theorem shows the $O(1/t)$ convergence rate of $N$-block ADMM \eqref{admm-N}. We omit the proof here for the sake of succinctness.

\begin{theorem}
Assume that
\[
\gamma\leq\min_{i=2,\cdots,N-1}\left\{\frac{2\sigma_{i}}{(2N-i)(i-1)\lambda_{\max}(A_{i}^{\top}A_{i})}, \frac{2\sigma_{N}}{(N-2)(N+1)\lambda_{\max}(A_{N}^{\top}A_{N})}\right\},
\]
where $\sigma_{i}$ is the strong convexity parameter of $f_{i}$, $i=2,\ldots,N$.  Let $(x_{1}^{k+1},x_{2}^{k+1},x_{3}^{k+1},\cdots,x_{N}^{k+1}, \lambda^{k+1})\in\Omega$ be generated by the $N$-block ADMM \eqref{admm-N}. For any integer $t>0$, we define
\begin{eqnarray*}
\bar{x}_{i}^{t}=\frac{1}{t+1}\sum\limits_{k=0}^{t}x_{i}^{k+1}, 1\leq i \leq N, \quad\quad
\bar{\lambda}^{t}=\frac{1}{t+1}\sum\limits_{k=0}^{t}\lambda^{k+1}.
\end{eqnarray*}
Then, for $\rho:=\|\lambda^*\|+1$, it holds that
\begin{eqnarray*}
\sum\limits_{i=1}^{N}(f_{i}(\bar{x}_{i}^{t})-f_{i}(x_{i}^{*}))+\rho\left\| \sum\limits_{i=1}^{N}A_{i}\bar{x}_{i}^{t}-b\right\|\leq\frac{\gamma}{2(t+1)}\sum\limits_{i=1}^{N-1}
\left\|\sum\limits_{m=i+1}^{N}A_{m}(x_{m}^{0}-x_{m}^{*})\right\|^{2}
+\frac{\rho^{2}+\|\lambda^0\|^2}{\gamma (t+1)}.
\end{eqnarray*}
Similarly as Theorem \ref{thm-ergodic-3}, this also implies that $N$-block ADMM \eqref{admm-N} converges with rate $O(1/t)$ in terms both error of objective function value and the residual of the equality constraints, i.e., it holds that
\[|f(\bar{u}^t)-f(u^*)| = O(1/t), \quad \mbox{ and } \quad \left\| \sum\limits_{i=1}^{N}A_{i}\bar{x}_{i}^{t}-b\right\|=O(1/t).\]
\end{theorem}

\section{Non-Ergodic Convergence Rate of ADMM}\label{sec:non-ergodic}

In this section, we prove an $o(1/k)$ non-ergodic convergence rate for ADMM \eqref{eq::update_x}-\eqref{eq::update_lambda}.


Let us first observe the following (see also Lemma 4.1 in \cite{Han-Yuan-note-2012}). Suppose at the $(k+1)$-th iteration of ADMM \eqref{eq::update_x}-\eqref{eq::update_lambda}, we have
\begin{equation}\label{kkt-reduced}
\left\{
\begin{array}{l}
A_{2}x_{2}^{k+1}-A_{2}x_{2}^{k}=0,\\
A_{3}x_{3}^{k+1}-A_{3}x_{3}^{k}=0,\\
A_{1}x_{1}^{k+1}+A_{2}x_{2}^{k+1}+A_{3}x_{3}^{k+1}-b=0.
\end{array}
\right.
\end{equation}
Then, \eqref{opt-x1-lambda}-\eqref{opt-x3-lambda} would immediately lead to
\[
\left\{
\begin{array}{ll}
 (x_{1}-x_{1}^{k+1})^{\top}\left[ g_{1}(x_{1}^{k+1})-A_{1}^{\top}\lambda^{k+1} \right] \geq 0,       & \forall x_{1}\in\XCal_{1},  \\
 (x_{2}-x_{2}^{k+1})^{\top}\left[ g_2(x_{2}^{k+1})-A_{2}^{\top}\lambda^{k+1} \right] \geq 0,          & \forall x_{2}\in\XCal_{2}, \\
 (x_{3}-x_{3}^{k+1})^{\top}\left[ g_{3}(x_{3}^{k+1})-A_{3}^{\top}\lambda^{k+1}\right] \geq 0, & \forall x_{3}\in\XCal_{3}.
\end{array}
\right.
\]
In other words, if \eqref{kkt-reduced} is satisfied, then $(x_1^{k+1},x_2^{k+1},x_3^{k+1},\lambda^{k+1})$ would have been already an optimal solution for \eqref{prob:P}.
It is therefore natural to introduce a residual for the linear system \eqref{kkt-reduced} as an optimality measure. Below is such a measure, to be denoted by $R_{k+1}$:
\begin{eqnarray}\label{def:Rk}
R_{k+1} := \| A_{1}x_{1}^{k+1}+A_{2}x_{2}^{k+1}+A_{3}x_{3}^{k+1}-b\|^{2}+2\| A_{2}x_{2}^{k+1}-A_{2}x_{2}^{k}\|^{2}+3\| A_{3}x_{3}^{k+1}-A_{3}x_{3}^{k}\|^{2}.
\end{eqnarray}
In the sequel, we will show that $R_k$ converges to $0$ at the rate $o(1/k)$. Note that this gives the convergence rate of ADMM \eqref{eq::update_x}-\eqref{eq::update_lambda} in non-ergodic sense.

We first show that $R_k$ is non-increasing.
\begin{lemma}\label{lemma5}
Assume $\gamma\leq\min\{\frac{\sigma_2}{\lambda_{\max}(A_{2}^{\top}A_{2})}, \frac{\sigma_3}{\lambda_{\max}(A_{3}^{\top}A_{3})}\}$. Let the sequence $\{x_1^k,x_2^k,x_3^k,\lambda^k\}$ be generated by ADMM \eqref{eq::update_x}-\eqref{eq::update_lambda}. It holds that $R_k$ defined in \eqref{def:Rk} is non-increasing, i.e.,
\begin{eqnarray}\label{Rk-decreasing}
R_{k+1} \leq R_{k}, \mbox{ $k=0,1,2,...$}. 
\end{eqnarray}
\end{lemma}

\begin{proof}
Letting $x_{1}=x_{1}^{k}$ in \eqref{opt-x1} yields,
\begin{eqnarray*}
\begin{aligned}
& (x_{1}^{k}-x_{1}^{k+1})^{\top}\left[g_{1}(x_{1}^{k+1})-A_{1}^{\top}\lambda^{k}+\gamma A_{1}^{\top}(A_{1}x_{1}^{k+1}+A_{2}x_{2}^{k}+A_{3}x_{3}^{k}-b)\right]\geq 0,
\end{aligned}
\end{eqnarray*}
with $g_1 \in \partial f_1$, which further implies that
\begin{eqnarray}
& & (x_{1}^{k+1}-x_{1}^{k})^{\top}g_{1}(x_{1}^{k+1}) \nonumber \\
& \leq & (x_{1}^{k}-x_{1}^{k+1})^{\top}(-A_{1}^{\top}\lambda^{k})+(x_{1}^{k}-x_{1}^{k+1})^{\top}\left[ \gamma A_{1}^{\top}(A_{1}x_{1}^{k+1}+A_{2}x_{2}^{k}+A_{3}x_{3}^{k}-b)\right]  \nonumber \\
& = & (A_{1}x_{1}^{k}-A_{1}x_{1}^{k+1})^{\top}(-\lambda^{k})+\gamma(A_{1}x_{1}^{k}-A_{1}x_{1}^{k+1})^{\top}(A_{1}x_{1}^{k+1}+A_{2}x_{2}^{k}+A_{3}x_{3}^{k}-b)\nonumber \\
& = & (A_{1}x_{1}^{k}-A_{1}x_{1}^{k+1})^{\top}(-\lambda^{k})+\frac{\gamma}{2}\left(\| A_{1}x_{1}^{k}+A_{2}x_{2}^{k}+A_{3}x_{3}^{k}-b\|^{2} \right. \nonumber \\
&   & - \left. \| A_{1}x_{1}^{k+1}+A_{2}x_{2}^{k}+A_{3}x_{3}^{k}-b\|^{2} -\| A_{1}x_{1}^{k}-A_{1}x_{1}^{k+1}\|^{2}\right), \label{2_1_2}
\end{eqnarray}
where the last equality is due to the identity \eqref{identity-4}.
Letting $x_{1}=x_{1}^{k+1}$ in \eqref{opt-x1-lambda} with $k+1$ changed to $k$ yields,
\begin{eqnarray*}
\begin{aligned}
& (x_{1}^{k+1}-x_{1}^{k})^{\top}\left[g_{1}(x_{1}^{k})-A_{1}^{\top}\lambda^{k}+\gamma A_{1}^{\top}A_{2}(x_{2}^{k-1}-x_{2}^{k})+\gamma A_{1}^{\top}A_{3}(x_{3}^{k-1}-x_{3}^{k})\right] \geq 0,
\end{aligned}
\end{eqnarray*}
which further implies that
\begin{eqnarray}
& & (x_{1}^{k}-x_{1}^{k+1})^{\top} g_{1}(x_{1}^{k}) \nonumber \\
& \leq & (x_{1}^{k+1}-x_{1}^{k})^{\top}(-A_{1}^{\top}\lambda^{k})+ \gamma (x_{1}^{k+1}-x_{1}^{k})^{\top}\left[A_{1}^{\top}A_{2}(x_{2}^{k-1}-x_{2}^{k})+A_{1}^{\top}A_{3}(x_{3}^{k-1}-x_{3}^{k})\right] \nonumber \\
& = & (A_{1}x_{1}^{k+1}-A_{1}x_{1}^{k})^{\top}(-\lambda^{k})+ \gamma (A_{1}x_{1}^{k+1}-A_{1}x_{1}^{k})^{\top}\left[ A_{2}(x_{2}^{k-1}-x_{2}^{k})+A_{3}(x_{3}^{k-1}-x_{3}^{k})\right] \nonumber \\
& \leq & (A_{1}x_{1}^{k+1}-A_{1}x_{1}^{k})^{\top}(-\lambda^{k})+ \frac{\gamma}{2}\left(\| A_{1}x_{1}^{k+1}-A_{1}x_{1}^{k}\|^{2} + \| A_{2}(x_{2}^{k-1}-x_{2}^{k})+A_{3}(x_{3}^{k-1}-x_{3}^{k})\|^{2}\right) \nonumber \\
& \leq & (A_{1}x_{1}^{k+1}-A_{1}x_{1}^{k})^{\top}(-\lambda^{k}) \nonumber \\
& & + \frac{\gamma}{2}\left(\| A_{1}x_{1}^{k+1}-A_{1}x_{1}^{k}\|^{2} + 2\| A_{2}(x_{2}^{k-1}-x_{2}^{k})\|^{2} +2\| A_{3}(x_{3}^{k-1}-x_{3}^{k})\|^{2}\right). \label{2_2_1}
\end{eqnarray}
Combining \eqref{2_1_2} and \eqref{2_2_1} gives
\begin{eqnarray}
& & (x_{1}^{k+1}-x_{1}^{k})^{\top}\left[g_{1}(x_{1}^{k+1})-g_{1}(x_{1}^{k})\right] \nonumber \\
&\leq & \frac{\gamma}{2}\left(\| A_{1}x_{1}^{k}+A_{2}x_{2}^{k}+A_{3}x_{3}^{k}-b\|^{2} - \| A_{1}x_{1}^{k+1}+A_{2}x_{2}^{k}+A_{3}x_{3}^{k}-b\|^{2} \right. \label{2} \\
& & + \left. 2\| A_{2}(x_{2}^{k-1}-x_{2}^{k})\|^{2}+2\| A_{3}(x_{3}^{k-1}-x_{3}^{k})\|^{2}\right).\nonumber
\end{eqnarray}

Letting $x_{2}=x_{2}^{k}$ in \eqref{opt-x2} yields,
\begin{eqnarray*}
\begin{aligned}
& (x_{2}^{k}-x_{2}^{k+1})^{\top}\left[ g_{2}(x_{2}^{k+1})-A_{2}^{\top}\lambda^{k}+\gamma A_{2}^{\top}(A_{1}x_{1}^{k+1}+A_{2}x_{2}^{k+1}+A_{3}x_{3}^{k}-b)\right] \geq 0,
\end{aligned}
\end{eqnarray*}
which further implies that
\begin{eqnarray}
& & (x_{2}^{k+1}-x_{2}^{k})^{\top}g_{2}(x_{2}^{k+1}) \nonumber \\
& \leq & (x_{2}^{k}-x_{2}^{k+1})^{\top}(-A_{2}^{\top}\lambda^{k})+(x_{2}^{k}-x_{2}^{k+1})^{\top}\left[ \gamma A_{2}^{\top}(A_{1}x_{1}^{k+1}+A_{2}x_{2}^{k+1}+A_{3}x_{3}^{k}-b)\right] \nonumber \\
& = & (A_{2}x_{2}^{k}-A_{2}x_{2}^{k+1})^{\top}(-\lambda^{k})+\gamma(A_{2}x_{2}^{k}-A_{2}x_{2}^{k+1})^{\top}(A_{1}x_{1}^{k+1}+A_{2}x_{2}^{k+1}+A_{3}x_{3}^{k}-b) \nonumber \\
& = & (A_{2}x_{2}^{k}-A_{2}x_{2}^{k+1})^{\top}(-\lambda^{k})+\frac{\gamma}{2}\left(\| A_{1}x_{1}^{k+1}+A_{2}x_{2}^{k}+A_{3}x_{3}^{k}-b\|^{2} \right. \nonumber \\
& & - \left. \| A_{1}x_{1}^{k+1}+A_{2}x_{2}^{k+1}+A_{3}x_{3}^{k}-b\|^{2} -\| A_{2}x_{2}^{k}-A_{2}x_{2}^{k+1}\|^{2}\right), \label{3_1_2}
\end{eqnarray}
where the last equality is due to the identity \eqref{identity-4}.
Letting $x_{2}=x_{2}^{k+1}$ in \eqref{opt-x2-lambda} with $k+1$ changed to $k$ yields,
\[
(x_{2}^{k+1}-x_{2}^{k})^{\top}\left[g_{2}(x_{2}^{k})-A_{2}^{\top}\lambda^{k}+\gamma A_{2}^{\top}A_{3}(x_{3}^{k-1}-x_{3}^{k})\right] \geq 0,
\]
which further implies that
\begin{eqnarray}
&      & (x_{2}^{k}-x_{2}^{k+1})^{\top} g_{2}(x_{2}^{k}) \nonumber \\
& \leq & (x_{2}^{k+1}-x_{2}^{k})^{\top}(-A_{2}^{\top}\lambda^{k})+ \gamma (x_{2}^{k+1}-x_{2}^{k})^{\top}\left[A_{2}^{\top}A_{3}(x_{3}^{k-1}-x_{3}^{k})\right] \nonumber\\
& = & (A_{2}x_{2}^{k+1}-A_{2}x_{2}^{k})^{\top}(-\lambda^{k})+ \gamma (A_{2}x_{2}^{k+1}-A_{2}x_{2}^{k})^{\top}(A_{3}x_{3}^{k-1}-A_{3}x_{3}^{k})\nonumber \\
& \leq & (A_{2}x_{2}^{k+1}-A_{2}x_{2}^{k})^{\top}(-\lambda^{k})+\frac{\gamma}{2}\left(\| A_{2}x_{2}^{k+1}-A_{2}x_{2}^{k}\|^{2}+\| A_{3}x_{3}^{k-1}-A_{3}x_{3}^{k}\|^{2}\right). \label{3_2_1}
\end{eqnarray}
Combining \eqref{3_1_2} and \eqref{3_2_1} gives
\begin{eqnarray}
& & (x_{2}^{k+1}-x_{2}^{k})^{\top}\left[g_{2}(x_{2}^{k+1})-g_{2}(x_{2}^{k})\right] \nonumber \\
& \leq & \frac{\gamma}{2}\left(\| A_{1}x_{1}^{k+1}+A_{2}x_{2}^{k}+A_{3}x_{3}^{k}-b\|^{2} - \| A_{1}x_{1}^{k+1}+A_{2}x_{2}^{k+1}+A_{3}x_{3}^{k}-b\|^{2} \right. \nonumber \\
& & \left. +\| A_{3}x_{3}^{k-1}-A_{3}x_{3}^{k}\|^{2}\right). \label{3}
\end{eqnarray}

Letting $x_3=x_3^k$ in \eqref{opt-x3-lambda} and $x_3=x_3^{k+1}$ in \eqref{opt-x3-lambda} with $k+1$ changed to $k$, and adding the two resulting inequalities, yields,
\begin{eqnarray}
& & (x_{3}^{k+1}-x_{3}^{k})^{\top}\left[g_{3}(x_{3}^{k+1})-g_{3}(x_{3}^{k})\right] \nonumber \\
& \leq & (x_{3}^{k+1}-x_{3}^{k})^{\top}(A_{3}^{\top}\lambda^{k+1}-A_{3}^{\top}\lambda^{k}) \nonumber \\
& = & (A_{3}x_{3}^{k+1}-A_{3}x_{3}^{k})^{\top}(\lambda^{k+1}-\lambda^{k}) \nonumber \\
& = & \gamma(A_{3}x_{3}^{k}-A_{3}x_{3}^{k+1})^{\top}\left(A_{1}x_{1}^{k+1}+A_{2}x_{2}^{k+1}+A_{3}x_{3}^{k+1}-b\right) \nonumber \\
& = & \frac{\gamma}{2}\left(\| A_{1}x_{1}^{k+1}+A_{2}x_{2}^{k+1}+A_{3}x_{3}^{k}-b\|^{2}-\| A_{1}x_{1}^{k+1}+A_{2}x_{2}^{k+1}+A_{3}x_{3}^{k+1}-b\|^{2} \right. \nonumber \\
&   & \left. -\| A_{3}x_{3}^{k}-A_{3}x_{3}^{k+1}\|^{2}\right), \label{4}
\end{eqnarray}
where the last equality is due to the identity \eqref{identity-4}.

Combining \eqref{2}, \eqref{3} and \eqref{4} yields,
\begin{eqnarray}
& & (x_{1}^{k+1}-x_{1}^{k})^{\top}\left[g_{1}(x_{1}^{k+1})-g_{1}(x_{1}^{k})\right]+ (x_{2}^{k+1}-x_{2}^{k})^{\top}\left[g_{2}(x_{2}^{k+1})-g_{2}(x_{2}^{k})\right] \nonumber \\
& & + (x_{3}^{k+1}-x_{3}^{k})^{\top}\left[g_{3}(x_{3}^{k+1})-g_{3}(x_{3}^{k})\right] \nonumber \\
& \leq & \frac{\gamma}{2}\left[\| A_{1}x_{1}^{k}+A_{2}x_{2}^{k}+A_{3}x_{3}^{k}-b\|^{2} - \| A_{1}x_{1}^{k+1}+A_{2}x_{2}^{k}+A_{3}x_{3}^{k}-b\|^{2}+2\| A_{2}(x_{2}^{k-1}-x_{2}^{k})\|^{2} \right. \nonumber \\
& & +2\| A_{3}(x_{3}^{k-1}-x_{3}^{k})\|^{2}+\| A_{1}x_{1}^{k+1}+A_{2}x_{2}^{k}+A_{3}x_{3}^{k}-b\|^{2}-\| A_{1}x_{1}^{k+1}+A_{2}x_{2}^{k+1}+A_{3}x_{3}^{k}-b\|^{2} \nonumber \\
& & +\| A_{3}x_{3}^{k-1}-A_{3}x_{3}^{k}\|^{2}+ \| A_{1}x_{1}^{k+1}+A_{2}x_{2}^{k+1}+A_{3}x_{3}^{k}-b\|^{2} \nonumber \\
& & \left. -\| A_{1}x_{1}^{k+1}+A_{2}x_{2}^{k+1}+A_{3}x_{3}^{k+1}-b\|^{2}-\| A_{3}x_{3}^{k}-A_{3}x_{3}^{k+1}\|^{2} \right] \nonumber \\
& = & \frac{\gamma}{2}\left[\| A_{1}x_{1}^{k}+A_{2}x_{2}^{k}+A_{3}x_{3}^{k}-b\|^{2}+2\| A_{2}(x_{2}^{k-1}-x_{2}^{k})\|^{2}+3\| A_{3}(x_{3}^{k-1}-x_{3}^{k})\|^{2} \right. \nonumber \\
& & \left. -\| A_{3}x_{3}^{k}-A_{3}x_{3}^{k+1}\|^{2}-\| A_{1}x_{1}^{k+1}+A_{2}x_{2}^{k+1}+A_{3}x_{3}^{k+1}-b\|^{2}\right] \nonumber \\
& = & \frac{\gamma}{2} \left[R_{k}-R_{k+1}+2\| A_{2}x_{2}^{k+1}-A_{2}x_{2}^{k}\|^{2}+2\| A_{3}x_{3}^{k}-A_{3}x_{3}^{k+1}\|^{2}\right]. \label{add-2-3-4}
\end{eqnarray}
Note that \eqref{f2-strong-2} and \eqref{f3-strong-2} imply that
\begin{eqnarray}\label{f2-f3-strong}
\begin{aligned}
& (x_{2}^{k+1}-x_{2}^{k})^{\top}\left[g_{2}(x_{2}^{k+1})-g_{2}(x_{2}^{k})\right] \geq \sigma_2 \| x_{2}^{k+1}-x_{2}^{k}\|^{2},  \\
& (x_{3}^{k+1}-x_{3}^{k})^{\top}\left[g_{3}(x_{3}^{k+1})-g_{3}(x_{3}^{k})\right] \geq \sigma_3 \| x_{3}^{k+1}-x_{3}^{k}\|^{2}.
\end{aligned}
\end{eqnarray}
Combining \eqref{add-2-3-4} and \eqref{f2-f3-strong}, and the fact that $\gamma\leq\min\left\{\frac{\sigma_2}{\lambda_{\max}(A_{2}^{\top}A_{2})}, \frac{\sigma_3}{\lambda_{\max}(A_{3}^{\top}A_{3})}\right\}$, it is easy to see that $R_{k+1} \leq R_k$ for $k=0,1,2,...$.
\end{proof}

We are now ready to present the $o(1/k)$ non-ergodic convergence rate of the ADMM \eqref{eq::update_x}-\eqref{eq::update_lambda}.

\begin{theorem}
Assume $\gamma\leq\min\left\{\frac{\sigma_2}{2\lambda_{\max}(A_{2}^{\top}A_{2})}, \frac{\sigma_3}{2\lambda_{\max}(A_{3}^{\top}A_{3})}\right\}$. Let the sequence $\{x_1^{k},x_2^k,x_3^k,\lambda^k\}$ be generated by ADMM \eqref{eq::update_x}-\eqref{eq::update_lambda}. Then $\sum_{k=1}^\infty R_k < +\infty$
and $R_k = o(1/k)$.
\end{theorem}

\begin{proof}
Combining \eqref{3} and \eqref{4} yields
\begin{eqnarray*}
& & (x_{2}^{k+1}-x_{2}^{k})^{\top}\left[g_{2}(x_{2}^{k+1})-g_{2}(x_{2}^{k})\right] + (x_{3}^{k+1}-x_{3}^{k})^{\top}\left[g_{3}(x_{3}^{k+1})-g_{3}(x_{3}^{k})\right] \\
& \leq & \frac{\gamma}{2}\left[\| A_{1}x_{1}^{k+1}+A_{2}x_{2}^{k}+A_{3}x_{3}^{k}-b\|^{2}-\| A_{1}x_{1}^{k+1}+A_{2}x_{2}^{k+1}+A_{3}x_{3}^{k}-b\|^{2} \right. \\
& & +\| A_{3}x_{3}^{k-1}-A_{3}x_{3}^{k}\|^{2}+ \| A_{1}x_{1}^{k+1}+A_{2}x_{2}^{k+1}+A_{3}x_{3}^{k}-b\|^{2}\\
& & \left. -\| A_{1}x_{1}^{k+1}+A_{2}x_{2}^{k+1}+A_{3}x_{3}^{k+1}-b\|^{2}-\| A_{3}x_{3}^{k}-A_{3}x_{3}^{k+1}\|^{2}\right]\\
& = & \frac{\gamma}{2}\left[ \| A_{1}x_{1}^{k+1}+A_{2}x_{2}^{k}+A_{3}x_{3}^{k}-b\|^{2}+\| A_{3}x_{3}^{k-1}-A_{3}x_{3}^{k}\|^{2}-\| A_{3}x_{3}^{k}-A_{3}x_{3}^{k+1}\|^{2} \right. \\
& & \left. -\| A_{1}x_{1}^{k+1}+A_{2}x_{2}^{k+1}+A_{3}x_{3}^{k+1}-b\|^{2} \right] \\
& = & \frac{\gamma}{2}\left[\|A_{1}x_{1}^{k+1}+A_{2}x_{2}^{k}+A_{3}x_{3}^{k}-b\|^2 +\| A_{3}x_{3}^{k-1}-A_{3}x_{3}^{k}\|^{2}-\| A_{3}x_{3}^{k}-A_{3}x_{3}^{k+1}\|^{2} - R_{k+1} \right. \\
& & \left. + 2\|A_2x_2^{k+1}-A_2x_2^k\|^2 + 3\|A_3x_3^k-A_3x_3^{k+1}\|^2 \right] \\
& \leq & \frac{\gamma}{2}\left[\|A_{1}x_{1}^{k+1}+A_{2}x_{2}^{k}+A_{3}x_{3}^{k}-b\|^2 +\| A_{3}x_{3}^{k-1}-A_{3}x_{3}^{k}\|^{2}-\| A_{3}x_{3}^{k}-A_{3}x_{3}^{k+1}\|^{2} - R_{k+1}\right] \\
& & + \gamma\lambda_{\max}(A_2^\top A_2)\| x_{2}^{k+1}-x_{2}^{k}\|^{2} + \frac{3\gamma}{2}\lambda_{\max}(A_3^\top A_3)\| x_{3}^{k+1}-x_{3}^{k}\|^{2}.
\end{eqnarray*}
Using \eqref{f2-f3-strong}
and the assumption that $\gamma\leq\min\left\{\frac{\sigma_2}{2\lambda_{\max}(A_{2}^{\top}A_{2})}, \frac{\sigma_3}{2\lambda_{\max}(A_{3}^{\top}A_{3})}\right\}$, we obtain
\begin{equation}\label{Rk-bounded}
R_{k+1} \leq \| A_{1}x_{1}^{k+1}+A_{2}x_{2}^{k}+A_{3}x_{3}^{k}-b\|^{2}+\| A_{3}x_{3}^{k-1}-A_{3}x_{3}^{k}\|^{2}-\| A_{3}x_{3}^{k}-A_{3}x_{3}^{k+1}\|^{2}.
\end{equation}

From the optimality conditions \eqref{kkt} and the convexity of $f$, it follows that
\begin{equation}\label{thm_4}
f(u^{*})-f(u^{k+1})\leq(x_{1}^{*}-x_{1}^{k+1})^{\top}(A_{1}^{\top}\lambda^{*})+(x_{2}^{*}-x_{2}^{k+1})^{\top}(A_{2}^{\top}\lambda^{*})
+(x_{3}^{*}-x_{3}^{k+1})^{\top}(A_{3}^{\top}\lambda^{*}).
\end{equation}
By combining \eqref{1} and \eqref{thm_4}, we have
\begin{eqnarray}
& & \left(\begin{array}{c} x_{1}^{*}-x_{1}^{k+1} \\ x_{2}^{*}-x_{2}^{k+1} \\ x_{3}^{*}-x_{3}^{k+1} \end{array} \right)^{\top}
\left(\begin{array}{c} A_{1}^{\top}(\lambda^{*}-\lambda^{k+1}) \\ A_{2}^{\top}(\lambda^{*}-\lambda^{k+1}) \\ A_{3}^{\top}(\lambda^{*}-\lambda^{k+1}) \end{array} \right) \nonumber \\
& & + \frac{1}{2\gamma}\|\lambda^{k}-\lambda^{k+1}\|^{2}+\frac{\gamma}{2}\left(\| A_{1}x_{1}^{*}+A_{2}x_{2}^{*}+A_{3}x_{3}^{k}-b\|^{2}-\| A_{1}x_{1}^{*}+A_{2}x_{2}^{*}+A_{3}x_{3}^{k+1}-b\|^{2}\right) \nonumber \\
& & +\frac{\gamma}{2}\left(\| A_{1}x_{1}^{*}+A_{2}x_{2}^{k}+A_{3}x_{3}^{k}-b\|^{2}-\| A_{1}x_{1}^{*}+A_{2}x_{2}^{k+1}+A_{3}x_{3}^{k+1}-b\|^{2}\right) \nonumber \\
& \geq & \frac{\gamma}{2}\| A_{1}x_{1}^{k+1}+A_{2}x_{2}^{k}+A_{3}x_{3}^{k}-b\|^{2}. \label{thm_5}
\end{eqnarray}
Note that the first term in \eqref{thm_5} is equal to
\begin{eqnarray*}
& &  -(A_{1}x_{1}^{k+1}+A_{2}x_{2}^{k+1}+A_{3}x_{3}^{k+1}-b)^{\top}(\lambda^{*}-\lambda^{k+1}) \\
& = & \frac{1}{\gamma}\left(\lambda^{k+1}-\lambda^{k}\right)^{\top}\left(\lambda^{*}-\lambda^{k+1}\right) \\
& = & \frac{1}{2\gamma}\left(\|\lambda^{*}-\lambda^{k}\|^{2}
-\|\lambda^{k+1}-\lambda^{k}\|^{2}-\|\lambda^{*}-\lambda^{k+1}\|^{2}\right).
\end{eqnarray*}
Therefore, \eqref{thm_5} can be rearranged as
\begin{eqnarray}
& & \frac{1}{\gamma^{2}}\left(\|\lambda^{*}-\lambda^{k}\|^{2}-\|\lambda^{*}-\lambda^{k+1}\|^{2}\right)+(\| A_{1}x_{1}^{*}+A_{2}x_{2}^{*}+A_{3}x_{3}^{k}-b\|^{2}-\| A_{1}x_{1}^{*}+A_{2}x_{2}^{*}+A_{3}x_{3}^{k+1}-b\|^{2}) \nonumber \\
& & +\left(\| A_{1}x_{1}^{*}+A_{2}x_{2}^{k}+A_{3}x_{3}^{k}-b\|^{2}-\| A_{1}x_{1}^{*}+A_{2}x_{2}^{k+1}+A_{3}x_{3}^{k+1}-b\|^{2}\right) \nonumber \\
& \geq & \| A_{1}x_{1}^{k+1}+A_{2}x_{2}^{k}+A_{3}x_{3}^{k}-b\|^{2}. \label{redidual-bounded}
\end{eqnarray}
By \eqref{Rk-bounded} and \eqref{redidual-bounded} we get that
\begin{eqnarray*}
& & \sum\limits_{k=1}^{\infty}R_{k+1} \\
&\leq & \sum\limits_{k=1}^{\infty}\left[\| A_{1}x_{1}^{k+1}+A_{2}x_{2}^{k}+A_{3}x_{3}^{k}-b\|^{2}+\| A_{3}x_{3}^{k-1}-A_{3}x_{3}^{k}\|^{2}-\| A_{3}x_{3}^{k}-A_{3}x_{3}^{k+1}\|^{2}\right] \\
&\leq & \sum\limits_{k=1}^{\infty} \| A_{1}x_{1}^{k+1}+A_{2}x_{2}^{k}+A_{3}x_{3}^{k}-b\|^{2} + \| A_{3}x_{3}^{0}-A_{3}x_{3}^{1}\|^{2} \\
&\leq & \| A_{3}x_{3}^{0}-A_{3}x_{3}^{1}\|^{2}+\sum\limits_{k=1}^{\infty} \left[ \left(\| A_{1}x_{1}^{*}+A_{2}x_{2}^{*}+A_{3}x_{3}^{k}-b\|^{2}-\| A_{1}x_{1}^{*}+A_{2}x_{2}^{*}+A_{3}x_{3}^{k+1}-b\|^{2}\right) \right. \\
& & +\left(\| A_{1}x_{1}^{*}+A_{2}x_{2}^{k}+A_{3}x_{3}^{k}-b\|^{2}-\| A_{1}x_{1}^{*}+A_{2}x_{2}^{k+1}+A_{3}x_{3}^{k+1}-b\|^{2}\right) \\
& & \left. + \frac{1}{\gamma^{2}}\left(\|\lambda^{*}-\lambda^{k}\|^{2}-\|\lambda^{*}-\lambda^{k+1}\|^{2}\right) \right] \\
& \leq & \| A_{3}x_{3}^{0}-A_{3}x_{3}^{1}\|^{2}+\| A_{1}x_{1}^{*}+A_{2}x_{2}^{*}+A_{3}x_{3}^{1}-b\|^{2}+\| A_{1}x_{1}^{*}+A_{2}x_{2}^{1}+A_{3}x_{3}^{1}-b\|^{2}+\frac{1}{\gamma^{2}}\|\lambda^{*}-\lambda^{1}\|^{2}.
\end{eqnarray*}
Note that we have proved that $R_k$ is monotonically non-increasing, and $\sum_{k=1}^\infty R_k<+\infty$. As observed in Lemma 1.2 of \cite{Deng-admm-2014}, one has
\[
k R_{2k} \le R_k + R_{k+1} + \cdots + R_{2k} \rightarrow 0, \mbox{ as } k \rightarrow \infty,
\]
and therefore $R_k = o(1/k)$.
\end{proof}

\begin{remark}
\blue{We remark here that using similar arguments, it is easy to see that \eqref{Rk-bounded} and \eqref{redidual-bounded} together with the monotonicity of $R_k$ also imply that $R_k$ has a non-asymptotic sublinear convergence rate $O(1/k)$.}
\end{remark}

Note that our analysis can be extended to $N$-block ADMM \eqref{admm-N} easily. The results are summarized in the following theorem and the proof is omitted for the sake of succinctness.
\begin{theorem}
Assume that
\[
\gamma\leq\min_{i=2,\cdots,N-1}\left\{\frac{2\sigma_{i}}{(2N-i)(i-1)\lambda_{\max}(A_{i}^{\top}A_{i})}, \frac{2\sigma_{N}}{(N-2)(N+1)\lambda_{\max}(A_{N}^{\top}A_{N})}\right\}.
\]
Let $(x_{1}^{k+1},x_{2}^{k+1},x_{3}^{k+1},\cdots,x_{N}^{k+1},\lambda^{k+1})\in\Omega$ be generated by ADMM \eqref{admm-N}. Then $\sum_{k=1}^\infty R_k < +\infty$ and $R_k=o(1/k)$,
where $R_k$ is defined as
\[R_{k+1} := \left\|\sum_{i=1}^N A_ix_i^{k+1}-b\right\|^2 + \sum_{i=2}^N\frac{(2N-i)(i-1)}{2}\|A_ix_i^k-A_ix_i^{k+1}\|^2. \]
\end{theorem}

\section{Conclusions}\label{sec:conclusion}

In this paper, we analyzed the sublinear convergence rate of the standard Gauss-Seidel multi-block ADMM in both ergodic and non-ergodic sense. These are the first sublinear convergence rate results for standard multi-block ADMM. Using the techniques developed in this paper, we can also analyze the convergence rate of some variants of the standard multi-block ADMM such as the ones studied in \cite{He-Hou-Yuan-Jacob-2013} and \cite{Deng-admm-2014}, where the primal variables are updated in a Jacobi manner;
we plan to pursue this direction of research in the future.

{
We remark here the techniques developed in this paper can lead to a very simple proof for the $O(1/t)$ complexity of two-block ADMM in terms of objective error and constraint violation of \eqref{prob:N} ($N=2$). Specifically, when $N=2$, denote $(x_1^k,x_2^k;\lambda^k)$ as the iterate generated by the two-block ADMM \eqref{admm-N}, and define
\[\bar{x}_1^t = \frac{1}{t+1}\sum_{k=0}^t x_1^{k+1}, \quad \bar{x}_2^t = \frac{1}{t+1}\sum_{k=0}^t x_2^{k+1}, \quad \bar{\lambda}^t = \frac{1}{t+1}\sum_{k=0}^t \lambda^{k+1}.\]
We can prove that
\be\label{2-admm-eps} |f_1(\bar{x}_1^t) + f_2(\bar{x}_2^t) - f_1(x_1^*) - f_2(x_2^*) | = O(1/t), \quad \mbox{ and } \quad \|A_1\bar{x}_1^t+A_2\bar{x}_2^t-b\|=O(1/t),\ee
i.e., the convergence rate of the two-block ADMM is $O(1/t)$ in terms of both objective error and constraint violation.
Note that for $N=2$, $\gamma$ can be any positive number and there is no need to impose the strong convexity on either $f_1$ or $f_2$. The proof of this result is as follows.

First, when $N=2$, the optimality conditions \eqref{opt-x1-lambda}-\eqref{opt-x3-lambda} reduce to
\begin{align}
& (x_{1}-x_{1}^{k+1})^{\top}\left[ g_{1}(x_{1}^{k+1})-A_{1}^{\top}\lambda^{k+1}
+\gamma A_{1}^{\top}A_{2}(x_{2}^{k}-x_{2}^{k+1}) \right] \geq 0, & \forall x_{1}\in\XCal_{1}, \label{2-admm-opt-x1-lambda} \\
& (x_{2}-x_{2}^{k+1})^{\top}\left[ g_{2}(x_{2}^{k+1})-A_{2}^{\top}\lambda^{k+1}\right] \geq 0, & \forall x_{2}\in\XCal_{2}. \label{2-admm-opt-x2-lambda}
\end{align}
Therefore, by letting $x_1=x_1^*$ in \eqref{2-admm-opt-x1-lambda}, $x_2=x_2^*$ in \eqref{2-admm-opt-x2-lambda}, and using the convexity of $f_1$ and $f_2$, we have
\[
\ba{ll}
 & f_1(x_1^{k+1}) - f_1(x_1^*) + f_2(x_2^{k+1}) - f_2(x_2^*) \\
\leq & g_1(x_1^{k+1})^\top(x_1^{k+1}-x_1^*) + g_2(x_2^{k+1})^\top(x_2^{k+1}-x_2^*) \\
\leq & (-A_1^\top\lambda^{k+1}+\gamma A_1^\top A_2(x_2^k-x_2^{k+1}))^\top (x_1^*-x_1^{k+1}) + (-A_2^\top\lambda^{k+1})^\top (x_2^*-x_2^{k+1}) \\
= & \frac{1}{\gamma}(\lambda^k-\lambda^{k+1})^\top\lambda^{k+1} + \gamma\left[(-A_2x_2^{k+1})-(-A_2x_2^k)\right]^\top\left[(A_1x_1^*-b)-(A_1x_1^{k+1}-b)\right] \\
= & \frac{1}{\gamma}(\lambda^k-\lambda^{k+1})^\top\lambda^{k+1} + \frac{\gamma}{2}\left(\|-A_1x_1^{k+1}+b-A_2x_2^{k+1}\|^2+\|A_1x_1^*-b+A_2x_2^k\|^2-\|-A_2x_2^{k+1}-A_1x_1^*+b\|^2\right.\\
 & \left.-\|A_1x_1^{k+1}-b+A_2x_2^k\|^2\right) \\
\leq & \frac{1}{\gamma}(\lambda^k-\lambda^{k+1})^\top\lambda^{k+1} + \frac{\gamma}{2}\left(\frac{1}{\gamma^2}\|\lambda^k-\lambda^{k+1}\|^2+\|A_1x_1^*-b+A_2x_2^k\|^2-\|-A_2x_2^{k+1}-A_1x_1^*+b\|^2\right),
\ea\]
where the second equality is due to \eqref{identity-4}. Thus for any $\lambda\in\br^p$, it holds that,
\be\label{2-admm-ineq-1}
\ba{ll}
 & f_1(x_1^{k+1}) - f_1(x_1^*) + f_2(x_2^{k+1}) - f_2(x_2^*) -\lambda^\top(A_1x_1^{k+1}+A_2x_2^{k+1}-b) \\
\leq & \frac{1}{\gamma}(\lambda^{k+1}-\lambda)^\top(\lambda^k-\lambda^{k+1}) + \frac{1}{2\gamma}\|\lambda^k-\lambda^{k+1}\|^2 + \frac{\gamma}{2}(\|A_1x_1^*+A_2x_2^k-b\|^2-\|A_1x_1^*+A_2x_2^{k+1}-b\|^2)\\
= & \frac{1}{2\gamma} (\|\lambda-\lambda^{k}\|^2 - \|\lambda-\lambda^{k+1}\|^2) + \frac{\gamma}{2}(\|A_1x_1^*+A_2x_2^k-b\|^2-\|A_1x_1^*+A_2x_2^{k+1}-b\|^2).
\ea
\ee
Summing \eqref{2-admm-ineq-1} over $k=0,1,\ldots,t$ yields,
\[\ba{ll}
& f_1(\bar{x}_1^{t}) - f_1(x_1^*) + f_2(\bar{x}_2^{t}) - f_2(x_2^*) -\lambda^\top(A_1\bar{x}_1^{t}+A_2\bar{x}_2^{t}-b) \\
\leq & \frac{1}{2\gamma(t+1)}\|\lambda-\lambda^0\|^2 + \frac{\gamma}{2(t+1)}\|A_1x_1^*+A_2x_2^0-b\|^2.
\ea\]
Based on the above bound, the error analysis for both the objective and the residual follow the same line of arguments as the proof of Theorem \ref{thm-ergodic-3}.
}

\section*{Acknowledgements}

We would like to thank the editor and the anonymous referees for carefully reading this paper and for insightful comments.

\bibliographystyle{plain}
\bibliography{admm3}

\end{document}